\documentclass[a4paper,conference]{IEEEtran}

\usepackage{amssymb}
\usepackage{cmupint}
\usepackage{hyperref}
\usepackage{mathtools}
\usepackage{parskip}

\usepackage{nimbusmononarrow}

\usepackage{graphicx}

\DeclareFontFamily{U}{matha}{\hyphenchar\font45}
\DeclareFontShape{U}{matha}{m}{n}{
      <5> <6> <7> <8> <9> <10> gen * matha
      <10.95> matha10 <12> <14.4> <17.28> <20.74> <24.88> matha12
      }{}
\DeclareSymbolFont{matha}{U}{matha}{m}{n}

\DeclareMathSymbol{\cap}           {2}{matha}{"58}
\DeclareMathSymbol{\cup}           {2}{matha}{"59}

\DeclareMathSymbol{\sim}           {3}{matha}{"12}

\DeclareFontFamily{U}{mathx}{\hyphenchar\font45}
\DeclareFontShape{U}{mathx}{m}{n}{
      <5> <6> <7> <8> <9> <10>
      <10.95> <12> <14.4> <17.28> <20.74> <24.88>
      mathx10
      }{}
\DeclareSymbolFont{mathx}{U}{mathx}{m}{n}

\DeclareMathDelimiter{\lbrace}     {4}{matha}{"74}{mathx}{"20}
\DeclareMathDelimiter{\rbrace}     {5}{matha}{"75}{mathx}{"28}

\everymath{\displaystyle}

\makeatletter
\AtBeginDocument{
  \check@mathfonts
  \setlength{\fontdimen17\textfont2}{\fontdimen16\textfont2}
}
\makeatother

\usepackage{contour}
\usepackage[normalem]{ulem}

\contourlength{0.9pt}

\newcommand{\myuline}[1]{%
  \uline{\phantom{#1}}%
  \llap{\contour{white}{#1}}%
}

\let\leq\leqslant
\let\geq\geqslant

\let\epsilon\varepsilon
\let\theta\vartheta

\usepackage[scr=boondoxupr]{mathalpha}
\usepackage[sans]{dsfont}

\newcommand{\B}{\mathsf{B}}
\newcommand{\C}{\mathsf{C}}
\newcommand{\E}{\mathds{E}}
\renewcommand{\H}{\mathsf{H}}
\renewcommand{\P}{\mathds{P}}
\newcommand{\Q}{\mathcal{Q}}
\newcommand{\R}{\mathds{R}}

\renewcommand{\b}{\mathscr{b}}
\newcommand{\w}{\mathscr{w}}

\newcommand{\sphere}{\mathbb{S}}

\usepackage{amsthm}

\theoremstyle{definition}

\newtheorem{manualtheoreminner}{Theorem}
\newenvironment{manualtheorem}[1]{%
  \IfBlankTF{#1}
    {\renewcommand{\themanualtheoreminner}{\unskip}}
    {\renewcommand\themanualtheoreminner{#1}}%
  \manualtheoreminner
}{\endmanualtheoreminner}

\newtheorem*{lem}{Lemma}
\newtheorem*{defn}{Definition}

\theoremstyle{remark}

\newtheorem*{rem}{Remark}

\title{Separating balls with partly random hyperplanes with a view to partly random neural networks}
\author{Olov Schavemaker, Utrecht University, the Netherlands}
\date{}

\begin{document}
\maketitle
\raggedright

\begin{abstract}
We derive exact expressions for the probabilities that partly random hyperplanes separate two Euclidean balls. The probability that a fully random hyperplane separates two balls turns out to be significantly smaller than the corresponding probabilities for hyperplanes which are not fully random in certain cases. Our results motivate studying partially random neural networks and provide a first step in this direction.
\end{abstract}

\section{Introduction}\label{sec:intro}

In this short paper we deduce exact expressions for the probabilities that three types of partly random hyperplanes separate two Euclidean balls. Letting $\H[w;b]$ be the zero locus of $u\mapsto(w\mid u)-b,$ we consider the following three types of partly random hyperplanes: $\H[w;\b],\H[\w;b],$ and $\H[\w;\b],$ where $w,b$ are deterministic and $\w,\b$ are random variates.

The hyperplane parametrization $\H[w;b]$ is redolent of neural networks by design; a ``neuron'' in a neural network consists of a nonlinear activation function being applied to a linear function of the form $u\mapsto(w\mid u)-b.$

This paper was inspired by \cite{sjoerd}, where the authors showed that random neural networks can make two fairly arbitrary subsets of Euclidean space linearly separable by covering those two sets with small enough balls. In \cite{sjoerd} they consider an architecture with two layers. Their result relies on letting the first layer be wide enough for it to separate all the balls covering one set from all of the balls covering the other set with high probability. This boils down to drawing enough fully random hyperplanes to separate all the balls covering one set from all the balls covering the other set with high probability. Clearly the probability one random hyperplane separates one pair of balls is of interest if one wants to get insight into how many random hyperplanes are needed to separate all pairs of balls with high probability.

Since sufficiently nice manifolds are known to admit nice covering number bounds \cite[Thm 4.2]{iwen}, \cite[Lemma B.1]{yap}, \\ \cite[{\S}5]{niyogi} the result in \cite{sjoerd} and possible improvements thereof could be of great interest to those interested in classifying manifolds; a commonly assumed situation.

Random neural networks might sound void of practical significance at first as one generally wants to learn the weights and biases. Well, learning the weights and biases can be slow, sensitive to hyperparameters, and final values may be suboptimal \cite{rvfl}.

One way to avoid these undesirabilities is to forgo learning some of the weights or/and biases; a popular architecture with this design philosophy is the so-called Random Vector Functional Link (RVFL) network \cite{rvfl}. In the case of RVFL networks, it has been known for quite a while that they are performant despite the random weights and biases \cite{duin,rawn}.

An RVFL consists of two layers, the first one being fully random and the second one being learned. As far as the author knows, partially random layers have not yet been studied despite having the potential to be a best of both worlds golden mean. In \S\ref{sec:discussion} we shall discuss how our theorems hint at this indeed being the case.

In \S\ref{sec:setting} we clarify the setting wherein we operate throughout the paper, including notation and assumptions. \S\kern-1pt\S\ref{sec:random_biases}--\ref{sec:fully_random} are each dedicated to a probability that a partly random hyper- plane separates two balls; $\H[w;\b],\H[\w;b],$ and $\H[\w;\b]$ in that order. Lastly, we discuss our results and their possible implications in \S\ref{sec:discussion}.

\section{Setting}\label{sec:setting}

Since we are interested in separating balls in $\R^n$ with hyperplanes, we introduce the following notation.

The Euclidean inner product and norm on $\R^n$ will be denoted by $(\cdot\!\mid\!\cdot)$ and $\lvert\mkern2mu\cdot\mkern2mu\rvert$ resp.

Let $\B[c,r]=\{u\in\R^n:\lvert c-u\rvert<r\},$ where $c\in\R^n$ and $r>0.$ Let $\H[w;b]$ be the zero locus of $u\mapsto(w\mid u)-b,$ where $b\in\R$ and $w\in\sphere^{n-1}=\{u\in\R^n:\lvert u\rvert=1\}.$

The notations $\B[c,r]$ and $\H[w;b]$ parametrize all balls and hyperplanes resp., albeit not quite uniquely, as $\H[w;b]=$ $\H[-w;-b].$
\begin{defn}
We say $\H[w;b]$ \myuline{separates} $\B[c,r]$ and $\B[x,p]$ if $(w\mid u)-b\gtrless 0$ for all $u\in\B[c,r]$ whereas $(w\mid u)-b\lessgtr 0$ for all $u\in\B[x,p].$
\end{defn}
The balls $\B[c,r]$ and $\B[x,p]$  clearly cannot be separated if $\B[c,r]\cap\B[x,p]\neq\varnothing.$ Since we are interested in partially random hyperplanes, we assume $\lvert c-x\rvert=p+r+\delta$ for some $\delta>0.$ If $\delta=0,$ only one hyperplane separates the balls $\B[c,r]$ and $\B[x,p];$ namely, $\{c-x\}^\perp+v,$ where $\{v\}=\partial\B[c,r]\cap\partial\B[x,p].$

When considering a hyperplane $\H[w;b],$ we shall refer to $w,b$ as the \myuline{weight} and \myuline{bias} resp., in keeping with neural networks. We account random weights uniformly on $\sphere^{n-1}$ and random biases uniformly on $[-k,k]$ with $k>0.$

Henceforth we shall always consider $\B[c,r]$ and $\B[x,p]$ so that $\lvert c-x\rvert=p+r+\delta$ for some $\delta>0.$ We withal assume that $n\geq 2$ and that $\w,\b$ are uniformly distributed on $\sphere^{n-1}$ and $[-k,k]$ resp., where $k\geq\lvert c\rvert\vee\lvert x\rvert>0~(\Leftarrow c\neq x).$

\section{Random biases}\label{sec:random_biases}

The first and easiest probability we are interested in is
\begin{align*}
\P_{\!\b}\{\exists w\in\sphere^{n-1}:\H[w;\b]\text{ separates }\B[c,r]\text{ and }\B[x,p]\}
\end{align*}
To compact notation, we will shorten this to
\begin{align*}
\P_{\!\b}\{\exists w:\H[w;\b]\text{ separates}\}
\end{align*}
Before we evaluate this probability, however, we return to the aforementioned $\H[w;b]=\H[-w;-b],$ which suggests that it may be more natural to consider $w\in\sphere^{n-1}\!/\{\pm\}\cong$ $\P^{n-1}$ instead, where $\P^{n-1}$ comprises lines in $\R^n$ through the origin. Indeed, the orthogonal complement of a hyper- plane is a line through the origin.

We therefore introduce a second parametrization of hyper- planes, namely $\H[a,v]=\{a\}^\perp+v,$ where $a\in\P^{n-1}$ and $v\in a$ is the unique point on the line $a$ wherethrough the hyperplane passes. This parametrization is unique and we have $\H[w;(w\mid v)]=\H[w\R,v]$ whenever $w\in\sphere^{n-1}$ and $v\in\R w.$

\begin{manualtheorem}{\ref*{sec:random_biases}}\label{thm:random_biases}
Let $\b$ be as in \S\ref{sec:setting}. Then
\begin{gather*}
\P_{\!\b}\{\exists w:\H[w;\b]\text{ separates}\}=\\
\P_{\!\b}\{\H[\omega\R,\omega\b]\text{ separates}\}=\delta/(2k)
\end{gather*}
where $\sphere^{n-1}\ni\omega$ is parallel to the line through $c$ and $x,$ which we will call $a.$
\end{manualtheorem}
\begin{proof}
Starting with the second equality, we can introduce coordinates onto $a$ by considering $\Omega:a\ni v\mapsto(\omega\mid v).$ Let the intervals $(s,t)$ and $(y,z)$ be the images of $\B[c,r]\cap a$ \& $\B[x,p]\cap a$ under $\Omega$ resp.; i.e., $s$ is the coordinate of one of the endpoints of the line segment $\B[c,r]\cap a,$ etc.

WLOG assuming that $t<y,$ we must have $y-t=\delta.$ Now, $\Omega(\omega\b)=\b$ so
\begin{align*}
\P_{\!\b}\{\H[\omega\R,\omega\b]\text{ separates}\}=\P_{\!\b}\{\b\in[t,y]\}=\delta/(2k)
\end{align*}
if $[t,y]\subseteq[-k,k];$ indeed, since $t,y\in(c,x),$
\begin{align*}
\lvert t\rvert\vee\lvert y\rvert\leq\lvert c\rvert\vee\lvert x\rvert\leq k
\end{align*}
As for the first equality, it is geometrically optimal to have $\H[w;\b]\perp a,$ which is achieved by letting $w=\omega.$ Indeed, each separating hyperplane must intersect $a$ with coordinate in the interval $[t,y]$ and every coordinate is achieved by at most one realization of $\b.$ Only $\H[\omega\R,\omega\b]$ can achieve the coordinates $t$ and $y,$ and since we have already shown it in fact achieves the entire interval, $w=\omega$ is optimal.

Since $(\omega\mid\omega\b)=\b\lvert\omega\rvert^2=\b$ the first equality follows.
\end{proof}
\fbox{Henceforth $a$ will always denote the line through $c$ and $x.$}

\section{Random weights}\label{sec:random_weights}

In this section we consider
\begin{align*}
\P_{\!\w}\{\exists b\in[-k,k]:\H[\w;b]\text{ separates}\}
\end{align*}
But first we recall two facts. First off, $\P^{n-1}$ is metrized by the angle between two lines $\measuredangle:\P^{n-1}\times\P^{n-1}\to[0,\tfrac{\pi}{2}].$ A closed metric ball on $\P^{n-1}$ is thus of the form
\begin{align*}
\C[\epsilon,\alpha]=\{\ell\in\P^{n-1}:\measuredangle(\ell,\epsilon)\leq\alpha\}
\end{align*}
which is a (double) cone with axis $\epsilon.$

Secondly, the \myuline{regularized beta function} is given by
\begin{align*}
I(\kappa;y,z)=\frac{1}{B(y,z)}\int_0^\kappa s^{y-1}(1-s)^{z-1}\,ds
\end{align*}
where $B(y,z)$ is the familiar beta function, $y,z>0,$ and $\kappa\in[0,1].$

\begin{manualtheorem}{\ref*{sec:random_weights}}\label{thm:random_weights}
Let $\w$ be as in \S\ref{sec:setting}. Then
\begin{gather*}
\P_{\!\w}\{\exists b:\H[\w;b]\text{ separates}\}=I(\Q;\tfrac{n-1}{2},\tfrac{1}{2})\\
=\P_{\!\w}\{\H[\w\R,(\w\mid v)\w]\text{ separates}\}
\end{gather*}
where $\Q=1-(\tfrac{p+r}{p+r+\delta})^2$ and $v=\tfrac{p}{p+r}c+\tfrac{r}{p+r}x.$
\end{manualtheorem}
\begin{rem}
The point $v$ is the vertex of the (unique) double cone touching both balls. The line $a$ is its axis.
\end{rem}
\begin{proof}
Let $\phi$ be the angle between the generatrices of the aforementioned double cone and $a.$ Then $\phi\in(0,\tfrac{\pi}{2})$ and $\sin\phi=\tfrac{p+r}{p+r+\delta}$ because
\begin{align}\label{eq:sin_phi}
\begin{split}
p+r+\delta=\lvert c-x\rvert=\lvert c-v\rvert&+\lvert v-x\rvert\\
=\frac{r}{\sin\phi}&+\frac{p}{\sin\phi}
\end{split}
\end{align}
using that tangent lines to circles are perpendicular to the radii at the points of tangency. WLOG let $v=0,$ so that $a\in\P^{n-1}$ for convenience.

The idea of the proof is to show that having the randomly oriented hyperplane pass through the vertex of the double cone is optimal. Separating the balls is then equivalent to separating the two nappes of the double cone.

Let $\alpha\in(0,\tfrac{\pi}{2}).$ We claim that
\begin{gather*}
\{\omega\in\P^{n-1}:\measuredangle(\ell,\omega)<\tfrac{\pi}{2}\text{ for all }\ell\in\C[a,\alpha]\}\\
=\C[a,\tfrac{\pi}{2}-\alpha]
\end{gather*}
Indeed, if $\ell\in\C[a,\alpha]$ and $\omega\in\C[a,\tfrac{\pi}{2}-\alpha],$ then
\begin{align*}
\measuredangle(\ell,\omega)\leq\measuredangle(\ell,a)+\measuredangle(a,\omega)\leq\alpha+(\tfrac{\pi}{2}-\alpha)=\tfrac{\pi}{2}
\end{align*}
Conversely, if $\omega\notin\C[a,\tfrac{\pi}{2}-\alpha],$ then $\measuredangle(\omega,\nu)<\alpha,$ with $\nu$ being the line through $v$ perpendicular to $a$ in the plane $a$ spans with $\omega.$ Letting $\ell$ be the $90^\circ$ rotation of $\omega$ about $v$ within the plane it spans with $a,$ plainly $\ell\in\C[a,\alpha]$ since $\measuredangle(\ell,a)=\measuredangle(\omega,\nu)$ and $\measuredangle(\ell,\omega)=\tfrac{\pi}{2}.$

Ergo, $\P_{\!\w}\{v+\{\w\}^\perp\text{ separates}\}$ equals the probability of separating the two nappes, which may be expressed as
\begin{gather}
\P_{\!\w}\{\measuredangle(\ell,\w\R)<\tfrac{\pi}{2}\text{ for all }\ell\in\C[a,\phi]\}\notag\\
=\P_{\!\w}\{\w\R\in\C[a,\tfrac{\pi}{2}-\phi]\}\label{eq:translated_cone_prob}
\end{gather}
because $\P^{n-1}\cong\sphere^{n-1}\!/\{\pm\}.$ However, \cite[(1)]{li} yields that
\begin{align*}
\P_{\!\w}\{\w\in\C[a,\tfrac{\pi}{2}-\phi]\}=2\times\tfrac{1}{2}I(\cos^2\phi;\tfrac{n-1}{2},\tfrac{1}{2})
\end{align*}
Using that $\cos^2\phi=1-\sin^2\phi,$ we have thus shown that
\begin{align*}
\P_{\!\w}\{v+\{\w\}^\perp\text{ separates}\}=I(1-(\tfrac{p+r}{p+r+\delta})^2;\tfrac{n-1}{2},\tfrac{1}{2})
\end{align*}
All that remains to be shown is therefore that
\begin{gather*}
\P_{\!\w}\{\exists b:\H[\w;b]\text{ separates}\}=\P_{\!\w}\{v+\{\w\}^\perp\text{ separates}\}\\
=\P_{\!\w}\{\H[\w\R,(\w\mid v)\w]\text{ separates}\}
\end{gather*}
As the intersection of $v+\{\w\}^\perp$ and $\w\R$ is $\w(\w\mid v),$ it follows that
\begin{align*}
v+\{\w\}^\perp=\H[\w\R,(\w\mid v)\w]=\H[\w;(\w\mid v)]
\end{align*}
That $b=(\w\mid v),$ or equivalently, $v\in\H[\w;b],$ is geometri- cally optimal may be seen as follows: if $q,$ the intersection of $a$ and $\H[\w;b],$ is WLOG closer to $\B[c,r]$ than $v$ is, then the cone touching $\B[c,r]$ with vertex $q$ has wider aperture than $\C[a,\phi],$ so the corresponding separating probability would be smaller because \eqref{eq:translated_cone_prob} is decreasing in $\phi.$

Lastly, $b=(\w\mid v)\Rightarrow\lvert b\rvert\leq\lvert\w\rvert.\lvert v\rvert\leq\lvert c\rvert\vee\lvert x\rvert\leq k.$
\end{proof}
\fbox{Henceforth $\Q,v$ will always be as in Theorem \ref{thm:random_weights}.}

\section{Fully random case}\label{sec:fully_random}

Lastly, we compute $\P_{\!\w,\b}\{\H[\w;\b]\text{ separates}\}.$
\begin{manualtheorem}{\ref*{sec:fully_random}}\label{thm:fully_random}
Let $\w,\b$ be as in \S\ref{sec:setting}. Then
\begin{gather*}
\P_{\!\w,\b}\{\H[\w;\b]\text{ separates}\}=\\
\frac{\lvert c-x\rvert}{2k}\biggl(\frac{\displaystyle\Q^{(n-1)/2}}{\tfrac{n-1}{2}B(\tfrac{n-1}{2},\tfrac{1}{2})}-\tfrac{p+r}{p+r+\delta}I(\Q;\tfrac{n-1}{2},\tfrac{1}{2})\biggr)
\end{gather*}
\end{manualtheorem}
\begin{proof}
WLOG let $v=0$ again. By Adam's law,
\begin{align*}
\P_{\!\w,\b}\{\H[\w;\b]\text{ separates}\}=\underset{\w}{\E}(\P_{\!\b}\{\H[\w;\b]\text{ separates}\mid\w\})
\end{align*}
Let $\alpha=\cos^{-1}\lvert(\w\mid\omega)\rvert,$ with $\omega$ as in Theorem \ref{thm:random_biases}. Then $\alpha=\measuredangle(a,\R\w).$ It follows from \eqref{eq:translated_cone_prob} that $\H[\w;(\w\mid v)],$ a fortiori $\H[\w;\b],$ separates only if $\alpha\leq\tfrac{\pi}{2}-\phi,$ where
\begin{align}\label{eq:phi}
\phi=\sin^{-1}\tfrac{p+r}{p+r+\delta}
\end{align}
Smaller $\alpha$ correspond to larger ranges of separating biases; e.g., if $\alpha=\tfrac{\pi}{2}-\phi,$ only $\b=(\w\mid v)$ separates.

Throughout the next paragraph we consider $\w$ fixed so that $\alpha\in(0,\tfrac{\pi}{2}-\phi].$ Note that $\alpha>0$ almost surely. We also let
\begin{align*}
U=\bigcup\{\H[\w;b]:b\in[-k,k]\text{ so that }\H[\w;b]\text{ separates}\}
\end{align*}
Using that $k\geq\lvert c\rvert\vee\lvert x\rvert>0,$ a computation similar to \eqref{eq:sin_phi} shows that the length of line segment $a\cap U$ is
\begin{align*}
\lvert c-x\rvert-\frac{p+r}{\sin(\tfrac{\pi}{2}-\alpha)}=\biggl(1-\frac{\sin\phi}{\sin(\tfrac{\pi}{2}-\alpha)}\biggr)\lvert c-x\rvert
\end{align*}
Indeed, letting $\mathsf{P}$ be the plane $\w$ spans with $a,$ the angle between $a$ and the line $\H[\w;\b]\cap\mathsf{P}$ is $\tfrac{\pi}{2}-\alpha.$ Projecting onto $\R\w$ yields that the length of $\R\w\cap U$ is
\begin{align*}
\biggl(1-\frac{\sin\phi}{\sin(\tfrac{\pi}{2}-\alpha)}\biggr)\lvert c-x\rvert\cos\alpha=(\cos\alpha-\sin\phi)\lvert c-x\rvert
\end{align*}
Ergo, using Iverson brackets, $\P_{\!\b}\{\H[\w;\b]\text{ separates}\mid\w\}=$
\begin{align*}
\frac{\lvert c-x\rvert}{2k}(\cos\alpha-\sin\phi)[\cos\alpha\geq\sin\phi]
\end{align*}
Plugging this back into the expectation, we get
\begin{align*}
\frac{\lvert c-x\rvert}{2k}\underset{\w}{\E}\Bigl(\Bigl(\lvert(\w\mid\omega)\rvert-\sin\phi\Bigr)\Bigl[\lvert(\w\mid\omega)\rvert\geq\sin\phi\Bigr]\Bigr)
\end{align*}
By the well-known Funk--Hecke formula \cite{chafai} the above expectation equals
\begin{gather*}
\frac{1}{B(\tfrac{n-1}{2},\tfrac{1}{2})}\int_{-1}^{+1}\Bigl(\lvert s\rvert-\sin\phi\Bigr)\Bigl[\lvert s\rvert\geq\sin\phi\Bigr](1-s^2)^{(n-3)/2}\,ds\\
=\frac{2}{B(\tfrac{n-1}{2},\tfrac{1}{2})}\int_{\sin\phi}^1(s-\sin\phi)(1-s^2)^{(n-3)/2}\,ds
\end{gather*}
The above integral evaluates to
\begin{align*}
\frac{\displaystyle\Q^{(n-1)/2}}{\tfrac{n-1}{2}B(\tfrac{n-1}{2},\tfrac{1}{2})}-I(\Q;\tfrac{n-1}{2},\tfrac{1}{2})\sin\phi
\end{align*}
where we recall that $\Q=1-\sin^2\phi.$ Putting everything together now yields the desideratum.
\end{proof}
It behooves us to verify that the resultant probability in Theorem \ref{thm:fully_random} is indeed a probability, since it is no longer immediately clear.
\begin{lem}
If $\alpha\in(0,\tfrac{\pi}{2}),$ then
\begin{gather*}
I(\cos^2\alpha;\tfrac{n-1}{2},\tfrac{1}{2})\sin\alpha\leq\\
\frac{\displaystyle\cos^{n-1}\alpha}{\tfrac{n-1}{2}B(\tfrac{n-1}{2},\tfrac{1}{2})}\leq I(\cos^2\alpha;\tfrac{n-1}{2},\tfrac{1}{2})
\end{gather*}
\end{lem}
\begin{proof}
Let $\kappa=\cos^2\alpha.$ Plainly $s\in[0,\kappa]\Rightarrow$
\begin{align*}
1\leq(1-s)^{-1/2}\leq(1-\cos^2\alpha)^{-1/2}=\csc\alpha
\end{align*}
Integration now yields that
\begin{gather*}
\int_0^\kappa s^{(n-3)/2}\,ds\leq\int_0^\kappa s^{(n-3)/2}(1-s)^{-1/2}\,ds\\
\leq\csc\alpha\int_0^\kappa s^{(n-3)/2}\,ds
\end{gather*}
The rest is trivial.
\end{proof}
Applying the Lemma with $\alpha=\phi$ as in \eqref{eq:phi} evinces that  the resultant probability in Theorem \ref{thm:fully_random} is indeed a probability, since $\lvert c-x\rvert\leq\lvert c\rvert+\lvert x\rvert\leq 2k$ as well. 

\section{Discussion}\label{sec:discussion}

In \S\ref{sec:intro} we suggested that partially random neural network layers could be a golden mean between fully learned and fully random layers. We shall therefore discuss how the separating probability for fully random hyperplanes com- pares to those for random weights and random biases.

Comparing the resultant probabilities in Theorems \ref{thm:random_weights} and \ref{thm:fully_random} is easy thanks to our Lemma:
\begin{gather*}
\frac{\lvert c-x\rvert}{2k}\biggl(\frac{\displaystyle\Q^{(n-1)/2}}{\tfrac{n-1}{2}B(\tfrac{n-1}{2},\tfrac{1}{2})}-\tfrac{p+r}{p+r+\delta}I(\Q;\tfrac{n-1}{2},\tfrac{1}{2})\biggr)\\
\leq\frac{\displaystyle\Q^{(n-1)/2}}{\tfrac{n-1}{2}B(\tfrac{n-1}{2},\tfrac{1}{2})}-\tfrac{p+r}{p+r+\delta}I(\Q;\tfrac{n-1}{2},\tfrac{1}{2})\\
\leq\frac{\displaystyle\Q^{(n-1)/2}}{\tfrac{n-1}{2}B(\tfrac{n-1}{2},\tfrac{1}{2})}\leq I(\Q;\tfrac{n-1}{2},\tfrac{1}{2})
\end{gather*}
which makes sense: choosing the bias optimally should improve the separating probability.

Of note is also that $\Q\in(0,1)\Rightarrow$
\begin{gather*}
\frac{\displaystyle\Q^{(n-1)/2}}{\tfrac{n-1}{2}B(\tfrac{n-1}{2},\tfrac{1}{2})}<\frac{\Gamma(\tfrac{n}{2})}{\tfrac{n-1}{2}\Gamma(\tfrac{n-1}{2})\Gamma(\tfrac{1}{2})}\\
=\frac{\Gamma(\tfrac{n}{2})}{\Gamma(\tfrac{n+1}{2})\sqrt{\pi}}\sim\sqrt{2/(n\pi)}
\end{gather*}
so $\P\{\H[\w;\b]\text{ separates}\}\preccurlyeq 1/\sqrt{n}$ where ``$\sim$'' \& ``$\preccurlyeq$'' are as in \cite[pp.\ 2f.]{hardy}; i.e., fully random hyperplanes separate balls really poorly if the ambient dimension is large regardless of $p,r,\delta.$

Note that this is not the case for hyperplanes with random biases (and optimal weights). There the separating probabi- lity can get as close to 1 as desired if $\tfrac{p+r}{p+r+\delta}$ is sufficiently small (compared to the ambient dimension).

Comparing the resultant probabilities in Theorems \ref{thm:random_biases} and \ref{thm:fully_random} is a bit harder. Recall that $\P\{\H[\w;\b]\text{ separates}\mid\w\}=$
\begin{gather*}
\frac{\cos\alpha}{2k}\biggl(\lvert c-x\rvert-\frac{p+r}{\sin(\tfrac{\pi}{2}-\alpha)}\biggr)[\cos\alpha\geq\sin\phi]\\
\leq\frac{1}{2k}\Bigl(\lvert c-x\rvert-p-r\Bigr)=\delta/(2k)
\end{gather*}
so $\P\{\H[\w;\b]\text{ separates}\}=\E(\P\{\H[\w;\b]\text{ separates}\mid\w\})\leq$ $\E(\delta/(2k))=\delta/(2k)=\P\{\exists w:\H[w,\b]\text{ separates}\},$ as is to be expected.

Like the resultant probability for random weights, the resul- tant probability for random biases is unlike the resultant probability for fully random hyperplanes in the sense that $\delta/(2k)\not\preccurlyeq 1/\sqrt{n}.$

All in all, we have seen that hyperplanes with random bi- ases (and optimal weights) or random weights (and optimal biases) are significantly better at separating small balls in a high dimensional space than fully random hyperplanes. But small balls in high dimensional space is exactly the situation we talked about in \S\ref{sec:intro}. This suggests that partially random neural networks may be quite a bit better at classifying low dimensional manifolds than their fully random counterparts. A detailed exploration of this topic is beyond the scope of this paper, however.

The author is currently cosupervising Guido Wagenvoorde (master student at Utrecht University) who is currently, i.a., numerically comparing RVFL, a neural net of the form
\begin{align}\label{eq:rvfl}
\R^n\ni x\mapsto \sum_{j=1}^m a_j\rho((w_j\mid x)+b_j)=(a\mid\varrho)\\
a=\{a_j\}_{j=1}^m,\varrho=\Bigl\{\rho((w_j\mid x)+b_j)\Bigr\}_{j=1}^m\notag
\end{align}
with random $w_j$'s and $b_j$'s, with RVFL with learned biases (random $w_j$'s only) in his thesis. Many popular activation functions $\rho,$ like ReLU \& (hard) tanh, are positive on posi- tive inputs and nonpositive on nonpositive inputs, thus in which cell of the hyperplane tessellation induced by the $\H[w_j;b_j]$'s $x$ lies determines in which octant $\varrho$ lies when- ever $\rho$ is positive for positive inputs and nonpositive for nonpositive inputs, evincing the linkage between \eqref{eq:rvfl} and hyperplane tessellations.

In light of the linkage between \eqref{eq:rvfl} and hyperplane tessella- tions, RVFL thus corresponds to fully random hyperplanes and RVFL with learned biases corresponds to hyperplanes with random weights (and optimal biases). Extending our results to partly random hyperplane tessellations would be \\ a nice follow-up direction, seeing as (random) hyperplane tessellation is an active area of research interesting in its own right \cite{hyperplane_tessellation}.

\section*{Acknowledgements}

The author wants to thank Guido Wagenvoorde for making plots especially for this paper, and apologize for not being able to fit them within the allotted four pages. I'm sorry.

The author would also like to thank Palina Salanevich for her feedback on the initial draft of this paper.


\appendix

Whilst the author was unable to fit Wagenvoorde's plot in the to-be conference proceedings version, the author felt it would be a shame not to expound on them in this version, hence this appendix.

Wagenvoorde and the author had four architectures of the form \eqref{eq:rvfl} memorize circular octads of 2D points as stand-ins for Euclidean balls, see Fig.\ \ref{fig}. In each case $m=180,$ the networks were trained using the stochastic gradient descent method Adam, and the decision boundary was obtained by taking the zero loci of the networks.

Comparing the plots in Fig. \ref{fig}, the purple circular octad around $(6,-1)$ is more cleanly separated by ReLU RVFL with learned biases than ReLU RVFL and the yellow circu- lar octad around $(0,-1)$ is more cleanly separated by hard- tanh RVFL with learned biases than hardtanh RVFL, thus weakly corroborating our findings.

In the Discussion we conjectured that RVFL with learned biases outperforms RVFL greatliest in the regime of high- dimensional balls with small radii based on our theorems. Unfortunately, we only considered low-dimensional (2D) balls with relatively large radii for our numerics, as high- dimensional setups are too computationally expensive.

Lastly, we would be remiss not to mention that the decision boundaries in Fig. \ref{fig} are the result of a single realization, thus their evidential value is limited. For a more complete numerical juxtaposition of partly random shallow neural network architectures, we refer the interested reader to Wagenvoorde's upcoming master thesis.

\begin{figure*}
\includegraphics[width=\textwidth]{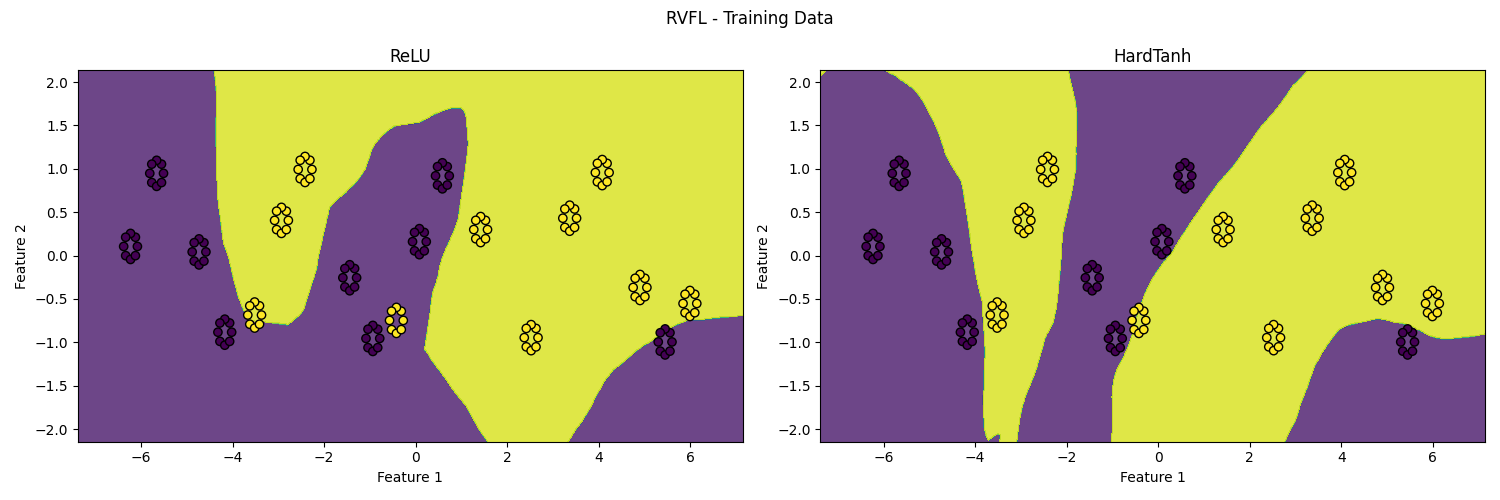}
\includegraphics[width=\textwidth]{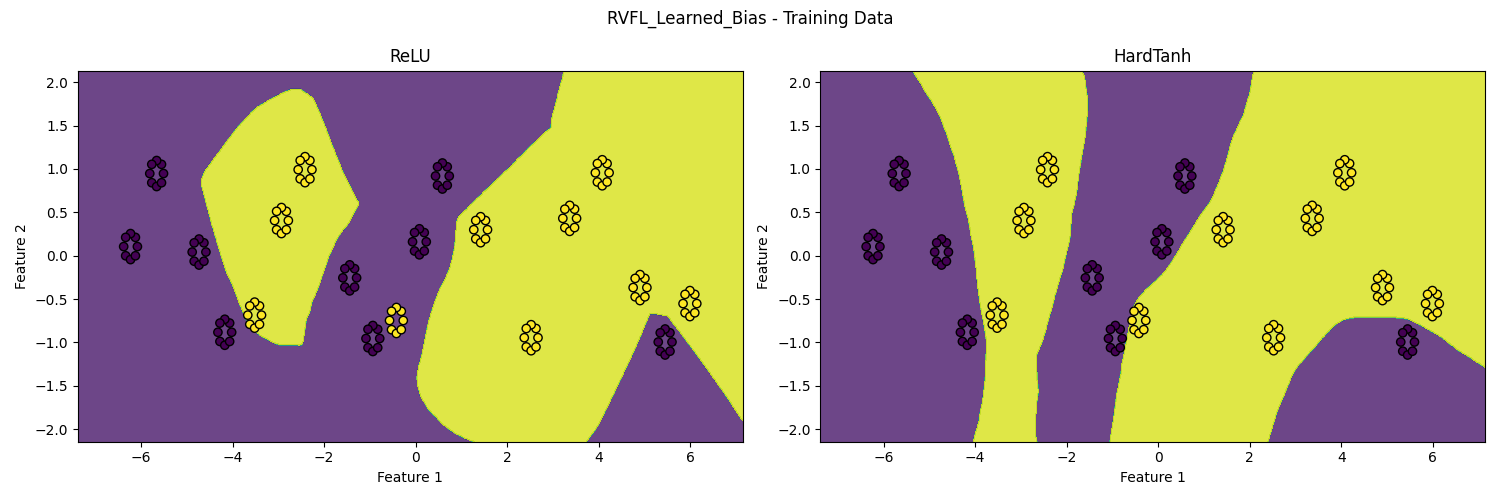}
\caption{\raggedright Binary classification decision boundaries for RVFL (top) and RVFL with learned biases (bottom) with ReLU (left) and hardtanh (right) activation functions. Figure courtesy of Guido Wagenvoorde, Mathematics Department, Utrecht University (master student).}
\label{fig}
\end{figure*}

\end{document}